\documentclass[a4paper,12pt]{article}
\usepackage[english]{babel}
\usepackage[tbtags]{amsmath}
\usepackage{amssymb}
\usepackage{amsfonts}
\usepackage{amsthm}
\usepackage{calrsfs}
\usepackage{array}
 \usepackage{bbm}
\usepackage{stmaryrd}
\usepackage{upgreek}
\usepackage[svgnames]{xcolor}
\usepackage{hyperref}
\usepackage{mathtools}
\def\pn{ p+1-\varepsilon}

\def\ppp{ \frac{p+3-\varepsilon}2}
\def\div{\, \mbox{div}\,  }

\def\d{\displaystyle}

\def\pp{ p-\varepsilon}
\def\w{\rho (y)}
\def\p{\phi }
\def\ps{\phi(s)}
\numberwithin{equation}{section}
\def\ibint {\int_{B}}
\def\iint {\int_B}

\def\grad{\partial_y}

\def \I {B(x_0,T(x_0)-t)}

\def\nn{\nu_0}

\def\v{{\mathrm{d}}v}

\def\Box{\hfill\rule{2.5mm}{2.5mm}}
\def\t{{\mathrm{d}}\tau}
\def\s{{\mathrm{d}}s}
\def\y{{\mathrm{d}}y}

\def\no{\nonumber}

\def\grad{\nabla}

\def \er {\mathbb R}

\def\R{{\mathbb {R}}}

\def\cprime{$'$}

\newcommand{\ds}{\displaystyle}

\def\t{{\mathrm{d}}\tau}
\def\s{{\mathrm{d}}s}
\def\y{{\mathrm{d}}y}

\def\y{{\mathrm{d}}y}
\def\N{\int_{s_1}^{s_2}\iint\! e^{-\frac{2(p+1)s}{p-1}}s^{\frac{2a}{p-1}} F(\p w)\w{\mathrm{d}}y{\mathrm{d}}s}

\theoremstyle{plain}
\newtheorem{thm}{Theorem}
\newtheorem*{thm*}{Theorem}

\newtheorem{prop}{Proposition}[section]
\newtheorem{cor}[prop]{Corollary}
\newtheorem{lem}[prop]{Lemma}

\theoremstyle{definition}

\theoremstyle{remark}
\newtheorem{nb}{Remark}[section]

\makeatletter
\def\blfootnote{\xdef\@thefnmark{}\@footnotetext}
\makeatother

\title{\bf The blow-up rate for    a non-scaling invariant   semilinear wave equations  in higher dimensions}

\author{Mohamed  Ali  Hamza\\
{\it \small 
Imam Abdulrahman Bin Faisal University,
 Dammam, 34212, Saudi Arabia}\\
Hatem Zaag\\
{\it \small Universit\'e Sorbonne Paris Nord},\\
{\it \small LAGA, CNRS (UMR 7539), F-93430, Villetaneuse, France}
}

\date{}

\allowdisplaybreaks

\begin{document} 

\newpage

\maketitle

\begin{abstract}
 We consider  the    semilinear wave
equation $$\partial_t^2 u -\Delta u =f(u), \quad  (x,t)\in \er^N\times [0,T),\qquad (1)$$ with
$f(u)=|u|^{p-1}u\log^a (2+u^2)$,  where $p>1$ and $a\in \er$,
with subconformal
power nonlinearity.
  We will show  that the blow-up rate of any singular solution  of (1)  is  given by the ODE solution  associated with $(1)$,
The result   in one space  dimension,   has been proved   in \cite{HZjmaa2020}.   Our goal here is to extend this result to higher dimensions.
\end{abstract}

\medskip



\section{Introduction}

\subsection{ Motivation of the problem }
This paper is devoted to the study of blow-up solutions for the
following semilinear  wave equation:
\begin{equation}\label{gen}
\left\{
\begin{array}{l}
\partial_t^2 u =\Delta u +f(u),\qquad  (x,t)\in \er^N\times [0,T),\\
\\
u(x,0)=u_0(x)\in  H^{1}_{loc,u}(\er^N),\qquad 
\partial_tu(x,0)=u_1(x)\in  L^{2}_{loc,u}(\er^N),
\end{array}
\right.
\end{equation}
where $u(t):x\in{\er^N} \rightarrow u(x,t)\in{\er}$ 
with focusing nonlinearity $f$ defined by:
\begin{equation}\label{deff}
 f(u)=   |u|^{p-1}u\log^a (2+u^2), \quad  p>1,\quad  a\in \er.
 \end{equation}
The spaces  $L^{2}_{loc,u}(\er^N)$ and $H^{1}_{loc,u}(\er^N)$ are  defined by
\begin{equation*}
L^{2}_{loc,u}(\er^N)=\{u:\er^N\rightarrow \er/ \sup_{d\in \er^N}(\int_{|x-d|\le 1}|u(x)|^2dx)<+\infty \},
\end{equation*}
and
\begin{equation*}
H^{1}_{loc,u}(\er^N)=\{u\in L^{2}_{loc,u}(\er^N),|\grad u|\in L^{2}_{loc,u}(\er^N) \}.
\end{equation*}
We assume in addition  that $p>1$  and if $N\ge 2$,  we further  assume that
\begin{equation}\label{subc}
p<p_c\equiv 1+\frac{4}{N-1}.
\end{equation}

\medskip

When $a\neq 0$, the nonlinearity in \eqref{deff} is not homogeneous, which means that equation \eqref{gen} is not scale invariant.
This is precisely our challenge in this paper, particularly in higher dimensions, since we handled the one
dimensional case in \cite{HZjmaa2020}.

\medskip

Semilinear wave equations with a nonlinearity showing a  logarithmic
factor  have been introduced
in  various nonlinear physical models, for instance in the context of nuclear physics,  wave mechanics, optics,  geophysics  etc ... see e.g. \cite{Bia75, Bia76}.

\medskip

  The defocusing case   has been  studied in the mathematical  literature and the first results
are due to 
Tao \cite{T1} where the author  proved a global well-posedness and scattering 
  result for   the three dimensional  nonlinear wave equation 
 $\partial_{t}^2 u =\Delta u-|u|^{4}u\log (2+u^2),$
in  the radial case. See also the 
work of Shih \cite{S1}, where the method is refined to treat   $\partial_{t}^2 u =\Delta u-|u|^{4}u\log^c (2+u^2),$ for any $c\in (0,\frac43)$. Later,  Roy   extends  in \cite{R1} the   results (global well-posedness and scattering)
to solutions of the log-log-supercritical equation
 $\partial_{t}^2 u =\Delta u-|u|^{4}u\log^{c} \big(\log (10+u^2)\big),$
for $c$ small,  without any radial assumption. 
Hoping to extend the validity of our argument to the conformal case  ($p=p_c$)
 in some forthcoming work, we may see the case $a > 0$  of \eqref{deff} as a further
step in the understanding of blow-up dynamics in the superconformal case related to
equation \eqref{NLW} below.

\medskip

Let us mention that  the blow-up question for  the   semilinear heat equation 
$ \partial_{t} u=\Delta u+|u|^{p-1}u\log ^a(2+u^2) $ 
 is  studied  by Duong-Nguyen-Zaag  in \cite{DVZ}. More precisely, they   construct for this equation a solution 
which blows up in finite time $T$,  only at one blow-up point $a$, according to the following asymptotic dynamics:
\begin{equation}\label{heatequiv}
u(x,t)\sim \phi(t)\Big(1+\frac{(p-1)|x-a|^2}{4p(T-t)|\log (T-t)|}\Big)^{-\frac{1}{p-1}},\qquad  as\  t\to T,
\end{equation}
  where $\phi(t)$ is is the unique positive solution of the ODE 
\begin{equation}\label{heat1}
\phi'=|\phi|^{p-1}\phi\log ^a(2+\phi^2),\ \qquad\qquad  \lim_{t
\to T}\phi(t)=+\infty.
\end{equation}
Given that we have the same expression in the pure power nonlinearity case ($g(u)=|u|^{p-1}u$)   with   $\phi(t)$ replaced by 
 $\kappa (T-t)^{-\frac1{p-1}}$  (see \cite{BKnonl94}), we see that the effect of the nonlinearity is all encapsulated in the ODE \eqref{heat1}.

\bigskip

Equation  $\eqref{gen}$ is well-posed in
$H^{1}_{loc,u}\times L^{2}_{loc,u}$. This follows from the finite
speed of propagation and the well-posedness in $H^{1}(\er^N)\times L^{2}(\er^N)$.
The existence of
blow-up solutions $u(t)$ of $\eqref{gen}$ follows from
 ODE techniques or the energy-based blow-up criterion by Levine \cite{Ltams74} (see
also   \cite{LT, STV,T2}).  More
blow-up results can be found in Caffarelli and Friedman
 \cite{CFarma85,  CFtams86}, Kichenassamy and Littman
 \cite{KL1cpde93, KL2cpde93}.
Numerical simulations of blow-up are given by Bizo\'n {\it et  al.}  (see \cite{Bjnmp01, BBMWnonl10, BCTnonl04, 
BZnonl09}). 

\medskip

If $u$ is an arbitrary blow-up solution of   \eqref{gen}, we define (see for example Alinhac \cite{Apndeta95}) a 1-Lipschitz curve $\Gamma=\{(x,T(x))\}$
such that the maximal influence domain $D$ of $u$ (or the domain of definition of $u$) is written as
\begin{equation}\label{defdu}
D=\{(x,t)\;|\; t< T(x)\}.
\end{equation}
The time
$\bar T=\inf_{x\in {\R^N}}T(x)$ and $\Gamma$ are called the blow-up time and the blow-up graph of $u$.
A point $x_0$ is non characteristic
if  there are
\begin{equation}\label{nonchar}
\delta_0\in(0,1)\mbox{ and }t_0<T(x_0)\mbox{ such that }
u\;\;\mbox{is defined on }{\mathcal C}_{x_0, T(x_0), \delta_0}\cap \{t\ge t_0\}
\end{equation}
where ${\cal C}_{\bar x, \bar t, \bar \delta}=\{(x,t)\;|\; t< \bar
t-\bar \delta|x-\bar x|\}$. If not, $x_0$ is said to be characteristic.

\bigskip

In this paper, we study the  blow-up rate of any singular solution of \eqref{gen}. Before going on, it is necessary
to mention  that
 the blow-up rate in the case
with pure  power nonlinearity 
\begin{equation}\label{NLW}
\partial^2_{t} u =\Delta  u+|u|^{p-1}u,   \,\,\,(x,t)\in \R^N \times [0,T),
\end{equation}
was studied by Merle and Zaag  in \cite{MZajm03, MZimrn05, MZma05}. More precisely, 
they proved  that if $u$  is a solution of \eqref{NLW}
with blow-up graph $\Gamma : \{ x\mapsto T(x)\}$ and $x_{0}$ is a
non-characteristic point, then,  for all $t\in
[\frac{3T(x_{0})}{4},T(x_{0})]$,
 \begin{eqnarray}\label{mzmz}
  0 < \varepsilon_{0}(p)\leq (T(x_{0})-t)^{\frac{2}{
   p-1}}\frac{\|u(t)\|_{L^{2}(\I )}}
{(T(x_{0})-t)^{\frac{N}2}}\\
  +(T(x_{0})-t)^{\frac{2}{ p-1}+1}\Big(\frac{\|\partial_{t} u(t)\|_{L^{2}(\I)}}
  {(T(x_{0})-t)^{\frac{N}2}}
  + \frac{\|\partial_x u(t)\|_{L^{2}(\I )}}{{(T(x_{0})-t)^{\frac{N}2}}}\Big)\leq K_0,\nonumber
  \end{eqnarray}
 where the constant $K_0$ depends only on  $p$ and on an upper bound on
 $T(x_{0})$, ${1}/{T(x_{0})}$, $\delta_{0}(x_{0})$, together with the
 norm of initial data
 in $H^{1}_{loc,u}(\R^N)\times L^{2}_{loc,u}(\R^N)$. Namely, the blow-up rate of any singular solution of \eqref{NLW} is given by the solution of the  associated ODE
$u'' =|u|^{p-1}u$.
 Note that this result about the blow-up rate  is valid   
 in the subconformal  and conformal case  ($1<p\le  p_c$).

\medskip

 In a series of papers,   Merle and Zaag \cite{MZjfa07, MZcmp08, MZajm11, MZisol10}
  (see also   C\^ote and Zaag \cite{CZcpam13}) give a full picture of  blow-up
  for solutions  of  equation \eqref{NLW} in one space dimension.
Among other results, Merle and Zaag proved that characteristic
points are isolated and that the blow-up set $\{(x,T(x))\}$ is
${\cal {C}}^1$ near non-characteristic points and corner-shaped near
characteristic points. 
In higher dimensions, the method used in  the one-dimensional case no
longer holds
 because there  is no   classification of selfsimilar solutions of equation \eqref{NLW} in the energy space. However, in the radial case outside the origin, Merle and Zaag  reduce to  the one-dimensional case with  perturbation and  obtain the same results as for $N=1$ (see \cite{MZbsm11}
and also
the extension by Hamza and Zaag in \cite{HZkg12}
to the Klein-Gordon equation and other
damped lower-order perturbations of equation
 \eqref{NLW}). Later, Merle and Zaag could address the higher dimensional case in the subconformal
case and prove the stability of the explicit selfsimilar solution  with respect to the blow-up point
and initial data (see 
\cite{MZods, MZods14}).
Considering the behavior of radial solutions at the origin, Donninger and Sch{\"o}rkhuber were able to prove the stability of the ODE solution $u(t) =\kappa_0 (p)(T - t)^{-\frac{2}{p-1}} $  in the lightcone
 with respect to small perturbations in initial data, in
 a stronger topology (see \cite{DSdpde12,DStams14, DScmp16, DSaihp17}). 
 Their approach is based in particular on a good understanding of the
 spectral properties of the linearized operator in self-similar
 variables, operator which is not self-adjoint. Later, thanks to
 suitable Strichartz estimates for the critical wave equation    in similarity variables, 
 Donninger proved in \cite{DDuke}
 the stability of the  solution of the ODE
with respect to small perturbations in initial data, in the energy space.
Let us also mention that 
Killip, Stoval and Vi\c san proved in \cite{KSVma14}
that in  superconformal and Sobolev subcritical range, an upper bound on the blow-up rate is available.
This was further refined by Hamza and Zaag in \cite{HZdcds13}.

\medskip

In \cite{HZjhde12, HZnonl12}, using a highly non-trivial perturbative method, we could obtain the blow-up rate for the
 Klein-Gordon equation  and more generally, for equation 
\begin{equation}\label{NLWP}
\partial_t^2 u =\Delta u+|u|^{p-1}u+f(u)+g(\partial_t u ),\,\,\,(x,t)\in \R^N \times [0,T),
 \end{equation}
under the assumptions  $|f(u)|\leq M(1+|u|^q)$ and $|g(v)|\leq M(1+|v|)$,   for some $M > 0$ and $q<p\le  \frac{N+3}{N-1}$. 
In fact, we proved a similar result to $\eqref{mzmz}$,
valid in the subconformal  and conformal case.
  Let us also mention that in  \cite{H1, omar1, omar2}, the
results obtained in  \cite{HZjhde12,HZnonl12}   were extended by Hamza and Saidi to the  
strongly perturbed equation  (\ref{NLWP})  with  $|f(u)|\leq M(1+|u|^p\log^{-a}(2+u^2))$,  for some $a > 1$, though keeping the  same condition in $g$. 
Very recently, Azaiez and Zaag derived in \cite{AZ19} the blow-up rate for equations
of the type
\[
\partial_t^2 u = a(x) (\partial_x^2 u + \frac{N-1}x \partial_x u)
+b(x) |u|^{p-1}u +f(u)+ g(x,t,\partial_x u , \partial_t u)
\]
where
  $|f(u)|\le M(1+|u|^q)$ with $q<p$,
  $|g(x,t,v,z)|\le M(1+|v| \sqrt{a(x)}+|z|)$, for some $M>0$,
and $a(x)$ is typically $|x|^\alpha$ with $\alpha$ enjoying a infinite
number of values converging to $2$.

\medskip

%
%
%

In  the previous works   \cite{H1, omar1, omar2,HZjhde12,HZnonl12}, we   consider a  class of
 perturbed equations   where the nonlinear term is  equivalent to the pure power $|u|^{p-1}u$   and we obtain  the  estimate 
 \eqref{mzmz}. This is due to the fact that the dynamics are governed
 by the ODE equation: $u'' =|u|^{p-1}u$
 and not influenced by perturbative terms.
Furthermore, our proof remains (non trivially) perturbative  with respect to the homogeneous PDE \eqref{NLW}, which is scale invariant.

\bigskip
This leaves unanswered an interesting question: Is the scale
invariance  property
crucial  in deriving the blow-up rate?

\bigskip
 
In fact we {\textit{had the impression}} that the answer was ``yes'', since the scaling invariance induces in  similarity  variables a PDE which is  autonomous in the unperturbed case \eqref{NLW}, and asymptotically autonomous in the perturbed case \eqref{NLWP}.

\bigskip

In this paper we {\textit{prove}} that the answer is ``no'' from the
example of the PDE \eqref{gen} with the non homogeneous nonlinearity \eqref{deff}. In fact, 
our situation is different from  \eqref{NLW} and  \eqref{NLWP},
in the sense that 
the  term
$|u|^{p-1}u\log^a (2+u^2)$  
 is playing a fundamental role in the dynamics of  the blow-up  solution of \eqref{gen}. More  precisely, we obtain an analogous  result to \eqref{mzmz} but with a logarithmic correction  as shown in \eqref{main1} below.  In fact, the blow-up rate  is
 given by the solution of the  following ordinary differential
 equation:
%
%
%
%
%
%
%
%
 %
\begin{equation}\label{v}
v_T'' (t)=  |v_T (t)|^{p-1}v_T(t) \log^{a}\big(v_T^2(t)  +2\big), \quad v(T)=\infty,
\end{equation}
which satisfies
 \begin{equation}\label{equivv}
 v_T(t) \sim \kappa_{a} \psi_T(t), \text{ as } t \to\ T,\quad \textrm{ where}\quad  \kappa_{a} =  \left(\frac{2^{1-2a}(p+1)}{(p-1)^{2-a}} \right)^{\frac1{p-1}},
 \end{equation}
and
 \begin{equation}\label{psi}
 \psi_T(t)=(T - t)^{-\frac{2}{p-1}} (-\log (T - t))^{-\frac{a}{p-1}}
 \end{equation}
(see Lemma A.2  in \cite{HZjmaa2020}).

\subsection{Strategy of the proof}
Going back to the equation under study in this paper (see \eqref{gen} and \eqref{deff}), we
 introduce the following similarity variables, 
   defined for all
  $x_0\in \er^N$, $T_0$ such that $0< T_0\le T(x_0)$ by:
\begin{equation}\label{scaling}
y=\frac{x-x_0}{{T_0-t}},\quad s=-\log (T_0-t),\quad u(x,t)=\psi_{T_0}(t)w_{x_0,T_0}(y,s).
\end{equation}
On may think that it would be more natural to replace $\psi_{T_0}
(t)$ by $v_{T_0}(t)$ in this definition, since the latter is an exact solution of 
the ODE \eqref{v},
unlike the former.
 That might be a good idea, however, as $v_{T_0}(t)$ has no explicit expression, the calculations will immediately become too complicated.
For that reason, we preferred to replace the non-explicit $v_{T_0}(t)$ by
its explicit equivalent $\psi_{T_0}(t)$ in \eqref{psi}. The fact that the
latter is only an approximate solution and not an exact solution of \eqref{v} will have
no effect on 
our analysis.

%
From (\ref{gen}), the  function $w_{x_0,T_0}$  (we write $w$ for
simplicity) satisfies the following equation for all $y\in B$ and
$s\ge  \max ( -\log T_0, 1)$,
where $B\equiv B(0,1)$ stands for
the unit ball of $\er^N$ and throughout the paper:
\begin{align}\label{A}
\partial_{s}^2w&=\frac{1}{\rho}\div(\rho \grad w-\rho(y.\grad w)y)+\frac{2a}{(p-1)s}y.\grad w-\frac{2p+2}{(p-1)^2}w+\gamma(s)w\nonumber\\
&-\Big(\frac{p+3}{p-1}-\frac{2a}{(p-1)s}\Big)\partial_s w
-2y.\grad \partial_{s}w+
e^{-\frac{2ps}{p-1}}s^{\frac{a}{p-1}} f(\ps w),
\end{align}
with 
$\rho (y)=(1-|y|^2)^{\alpha}$,
\begin{equation}\label{alpha}
\alpha=\frac{2}{p-1}-\frac{N-1}{2}>0,
\end{equation}
\begin{equation}
   \gamma(s)=\frac{a(p+5)}{(p-1)^2s}-\frac{a(p+a-1)}{(p-1)^2s^2},\label{defgamma}
\end{equation}
and 
\begin{equation}
\ps =e^{\frac{2s}{p-1}}s^{-\frac{a}{p-1}}.\label{defphi}
\end{equation}
%
%
In the new
set of variables $(y,s),$ the behavior of $u$ as $t \rightarrow T_0$
is equivalent to the behavior of $w$ as $s \rightarrow +\infty$.  Moreover, if $T_0=T(x_0)$, then we simply write $w_{x_0}$
instead of $w_{x_0,T(x_0)}$.

\medskip

 The equation (\ref{A}) will be studied in the Hilbert  space $\cal H$
$${\cal H}=\Big \{(w_1,w_2), |
\displaystyle\ibint\Big(\big ( w_1^2 +|\grad w_1|^2-|y.\grad w_1|^2\big)+w_2^2\Big) \y<+\infty \Big \}.$$

\medskip

%
%
As in the pure power case \eqref{NLW} and the perturbed case
\eqref{NLWP},  the construction of a Lyapunov
functional in similarity variables was the starting point of our
strategy. In the present case \eqref{gen}, we adopt the same
strategy. We were successful in implementing that in
\cite{HZjmaa2020}, however, only in one space dimension.
 Indeed, 
 our method in  \cite{HZjmaa2020}  breaks down  in higher dimensions. 
Let us briefly
explain in the following the  method  used in  \cite{HZjmaa2020} and
how
it
breaks down  in higher dimensions,  giving sense to the present work. 

\medskip

The first step in \cite{HZjmaa2020} consists in the  introduction of a
functional associated to equation \eqref{A} 
 which satisfies the following  differential inequality:
 $$
\frac{d}{ds}h(s)\le -\alpha \ibint (\partial_{s}w)^2\frac{\w}{1-|y|^2}\y
+ \frac{C}{s}h(s),$$ where $\alpha$ is defined in \eqref{alpha} and
$w$  is the solution of   \eqref{A}.
Thanks to 
the above-mentioned functional, we 
easily derived a polynomial (in $s$)
bound for
the $H^1\times L^2(B)$ norm 
of   the solution of  \eqref{A} (in \textbf{space-time} averages). More precisely, we obtain the estimates
\eqref{feb19}, \eqref{feb191} and \eqref{feb192} below.
Let us recall that  the nonlinear term $f(u)$ given by \eqref{deff} is not a pure power. This is why   the strategy used    to remove the time averages  in the  case of pure power  in  Merle and Zaag  \cite{MZajm03, MZimrn05, MZma05},   and naturally
implemented  in our previous papers  \cite{H1, omar1, omar2,
  HZjhde12,HZnonl12}  in the perturbative cases,   
 breaks down in higher dimensions. Indeed,   this method is somehow based on 
some interpolation results in Sobolev spaces,
and
some
critical Gagliardo-Nirenberg estimates. However, in one dimension,  
 the strategy used    to remove the time  average works  since   it is based 
 on the embedding
$H^1(\R\times[-\log T,+\infty))\hookrightarrow L^{q}(\R\times[-\log T,+\infty))  $,
for any $q>1$.
Using
the
polynomial (in $s$)    bound for the $H^{1}_{loc,u}(\er)-$norm of  the solution of   \eqref{A} and the embedding  $H^1(\R)\hookrightarrow 
L^{\infty}(\R)$,  we derive  a  Lyapunov functional for
equation \eqref{A} in
one space dimension,
which is  a crucial step to obtain the
optimal estimate.

\bigskip

Since the embedding of $H^1$ into $L^q$ for any $q>1$ is specific to
dimension $1+1$ and doesn't hold in dimension $N+1$,
the    higher dimensional case requires new ideas,
which we explain in the following.

\medskip

First, we  recall  in \eqref{feb19}, \eqref{feb191} and \eqref{feb192}
the rough polynomial (in $s$) bound (in space-time
averages) on the solution near any non characteristic point   in
similarity variables,  was proved and stated  in any 
dimensions   in the subconformal case ($p<p_c$).  Moreover, we
can somehow reduce to the pure power case, up to an $\varepsilon$ 
perturbation, as we write in the elementary estimates on the nonlinear
term given below in \eqref{equiv1},  \eqref{equiv44} and
\eqref{equiv44bis}. In fact, by exploting  these estimates, 
we prove an improved version of  the estimates
\eqref{feb19} and \eqref{feb191}, where we remove the time average.
Then, 
 using    these new  estimates,   the embedding of  $H^1(\R^N)$ in 
$L^{2^*}(\R^N)  $ if $N\ge 3$ and in $L^q(\R^N)$ for any $q\ge 2$ if
$N=2$, together with the structure of the
 nonlinear term,  we end up with the   construction of a functional  $g(s)$
which  satisfies the following differential inequality:
 \begin{equation}\label{jardin1}
\frac{d}{ds}g(s)\le -\alpha \ibint (\partial_{s}w)^2\frac{\w}{1-|y|^2}\y
+ \frac{C}{s^{\frac32}}g(s)+ \frac{C}{s^{\frac74}}, 
\end{equation} 
where $\alpha$ is defined in \eqref{alpha} and $w$  is the solution of  \eqref{A}.
Naturally, by \eqref{jardin1}, we easily
derive 
  a  Lyapunov functional for equation \eqref{A}, valid in any dimensions in the subconformal 
case,  and this is   our main contribution in this work. With this
Lyapunov functional at hand, the adaptation of the interpolation
strategy from our previous papers works straightforwardly.

\subsection{
Statement of the results}
To state our main result,  we start by introducing the following 
functionals:
\begin{align}
  E(w(s),s)=&\iint
             \Big(\frac{1}{2}(\partial_{s}w)^2+\frac{1}{2}(|\grad
             w|^2-(y.\grad w)^2)+\frac{p+1}{(p-1)^2}w^2\nonumber\\
  &-e^{-\frac{2(p+1)s}{p-1}}s^{\frac{2a}{p-1}}   F(\p w)\Big)\w \y, \label{E}\\
L_0(w(s),s)=& \;E(w(s),s)-\frac1{s\sqrt{s}}
\iint \partial_{s}ww\w\y,\label{5jan1}
\end{align}
where
\begin{equation}\label{defF}
 F(u)=\int_{0}^{u}f(v){\mathrm{d}}v=\int_0^u|v|^{p-1}v \log^{{a}}(v^2  + 2)\v.
\end{equation}
%
Moreover,  for all
 $s\ge  \max (1, -\log T_0)$,  we define the functional
\begin{equation}\label{10dec2}
L(w(s),s)=\exp\Big(\frac{p+3}{\sqrt{s}}\Big) L_0(w(s),s)+\theta s^{-\frac34},
\end{equation}
where $\theta$ is a sufficiently large constant that will be determined later.
We will show that the functional  $L(w(s),s)$ is a decreasing 
  functional  of time  for equation (\ref{A}),  provided that $s$ is  large enough.
Clearly, by  \eqref{5jan1} and \eqref{10dec2},  the  functional $L(w(s),s)$  is a small perturbation of the   ``natural'' energy $E(w(s),s)$.

\medskip

Here is  the statement of  our main theorem in this paper.
 \begin{thm}\label{t1}
Consider   $u $    a solution of ({\ref{gen}}) with blow-up graph
$\Gamma:\{x\mapsto T(x)\}$, and  $x_0$  a non characteristic point.
Then there exists $t_1(x_0)\in [0,T(x_0)) $ such that, 
 for all $T_0\in  (t_1(x_0),T(x_0)]$,  for all  $s\ge  -\log(T_0-t_1(x_0))$, we have
\begin{equation*}
 L(w(s+1),s+1)-L(w(s),s) \leq -\alpha \int_{s}^{s+1}\iint (\partial_{s}w)^2\frac{\w}{1-|y|^2}\y \t,
\end{equation*}
where  $w=w_{x_0,T_0}$ is defined in \eqref{scaling}.
\end{thm}

\medskip

\begin{nb}
  Since
  the existence of a Lyapunov functional in similarity variables is far from being trivial   and represents the crucial  step
  in this paper, we choose to state it first in our paper, and to give it
  the status of a ``first theorem''.
\end{nb}
\begin{nb}
Let us note that 
our method breaks down in the  case of a   characteristic point, since in the construction of the 
Lyaponov functional 
in similarity variables,  we use  a covering technique in our argument which is not available  at a  characteristic point.
At this moment,  we do not know whether Theorem \ref{t1} continues to hold if  $x_0$ is a   characteristic point.
\end{nb}

As we said earlier, the existence of this Lyapunov functional  $ L(w(s),s)$ together with a blow-up criterion for
equation \eqref{A} make  a crucial step in the derivation of the blow-up
rate for equation \eqref{gen}. Indeed, with the functional $ L(w(s),s)$ and some more work, we are able to
adapt the analysis performed in  \cite{MZajm03, MZimrn05,MZma05} 
 for equation \eqref{NLW} and obtain the following result:

\begin{thm}\label{t2}
{\bf {(Blow-up rate for equation \eqref{gen})}}.\\
Consider   $u $    a solution of ({\ref{gen}}), with blow-up graph
$\Gamma:\{x\mapsto T(x)\}$ and  $x_0$  a non characteristic point.
Then, there exists  $\widehat{S}_2$  large enough  such that

i)
 For all
 $s\ge \widehat{s}_2(x_0)=\max(\widehat{S}_2,-\log \frac{T(x_0)}4)$,
\begin{equation*}
0<\varepsilon_0\le \|w_{x_0}(s)\|_{H^{1}(B)}+ \|\partial_s
w_{x_0}(s)\|_{L^{2}(B)} \le K,
\end{equation*}
where $w_{x_0}=w_{x_0,T(x_0)}$ is defined in (\ref{scaling}).\\
ii)  For all
  $t\in [t_2(x_0),T(x_0))$, where  $t_2(x_0)=T(x_0)-e^{-\widehat{s}_2(x_0)}$, we have
\begin{align}
&&0<\varepsilon_0\le \frac1{\psi_{T(x_0)}(t)}\frac{\|u(t)\|_{L^2(B(x_0,{T(x_0)-t}))}}{ {(T(x_0)-t)^{\frac{N}2}}}\label{main1}\\
&&+ \frac{T(x_0)-t}{\psi_{T(x_0)}(t)}\Big
(\frac{\|\partial_tu(t)\|_{L^2(B(x_0,{T(x_0)-t}))}}{
{(T(x_0)-t)^{\frac{N}2}}}+
 \frac{\|\partial_x u(t)\|_{L^2(B(x_0,{T(x_0)-t}))}}{ {(T(x_0)-t)^{\frac{N}2}}}\Big )\le K,\nonumber
\end{align}
where 
$K=K(p,  a, T(x_0), t_2(x_0),\|(u(t_2(x_0)),\partial_tu(t_2(x_0)))\|_{
H^{1}\times
L^{2}(B(x_0,\frac{T(x_0)-t_2(x_0)}{\delta_0(x_0)}) )})$, \\  $\psi_{T(x_0)}(t)$ is defined in \eqref{psi}  and $\delta_0(x_0)$ is defined in  \eqref{nonchar}.
\end{thm}

\begin{nb}
  Both for the construction of the Lyapunov  functional and the
  derivation of the bounds in this theorem, 
  our method breaks down in the conformal case,
even when  $a<0$,  and we are not able to obtain the sharp estimate
  as in the case of a  pure power nonlinearity
  \eqref{NLW} treated in \cite{MZma05}. 
\end{nb}

\begin{nb}
Since we crucially need a covering technique in the  argument of the construction of the Lyapunov functional,
our method breaks down too in the case  of a characteristic point  and we are not able to obtain the sharp estimate
as in the unperturbed case \eqref{NLW}. 
\end{nb}
\begin{nb}
As in  \cite{MZimrn05}  in the pure power nonlinearity  case   \eqref{NLW}, the proof of
Theorem \ref{t2} relies on four ideas: the existence of a Lyapunov functional,  interpolation in Sobolev spaces, some
critical Gagliardo-Nirenberg estimates and
a covering technique adapted to the geometric shape of the blow-up surface.
As we said before,  the first point where  we construt  a Lyapunov functional in similarity variables 
 is far from being trivial 
and represent  a  crucial  step. Consequently, 
we have chosen to present our main constribution as Theorem \eqref{t1}  and we 
 write a detailed  proof.  However,   for the other three points, the adaption of 
the proof of
\cite{MZimrn05} given in the  pure power nonlinearity  case \eqref{NLW} is straightforward   except for a key argument, where we
bound the nonlinear term  $e^{-\frac{2(p+1)s}{p-1}}s^{\frac{2a}{p-1}} F(\ps w)$.
Therefore, 
 in order to avoid unnecessary repetition, we prove this step and  kindly  refer to \cite{MZajm03, MZimrn05, MZma05, HZjhde12, HZnonl12,H1, omar1, omar2}  for the rest of the proof.
\end{nb}
\begin{nb}
Let us remark that we can  obtain the same  blow-up rate for  the  more general equation 
\begin{equation}\label{NLWlog}
\partial_t^2 u =\Delta u+|u|^{p-1}u\log^a(2+u^2)+k(u),\,\,\,(x,t)\in \R^N \times [0,T),
 \end{equation}
under the assumption  that $|k(u)|\leq M(1+|u|^{p}\log^b(2+u^2))$,   for some $M > 0$ and $b<a-1$. 
More precisely,   under this hypothesis,   we can construct  a suitable Lyapunov
functional for  this equation. Then, we  can prove a similar result to $\eqref{main1}$. However, the case where $a-1\le b<a$ seems 
to be out of reach with our techniques, though we think  we may obtain the same rate as in the unperturbed  case.
\end{nb}

%

This paper is organized as follows: In Section \ref{section2},  we
obtain a polynomial (in $s$) bound for the
$H^1\times L^2(B)$ norm
of the solution $(w, \partial_s w)$.
In Section \ref{section3}, thanks to this result, we prove that the 
functional $L(w(s),s)$ defined in \eqref{10dec2} is a Lyapunov functional for equation \eqref{A}. Thus, we get Theorem \ref{t1}. Finally, 
applying this last theorem,  we prove
Theorem \ref{t2}.

\medskip

Throughout this paper,
$C$  denotes a  generic positive constant
 depending only on $p,N$  and $a,$  which may vary from line to line.
 In addition, we  will use $K_1,K_2,K_3...$  as  a  positive constants
 depending only on $p,N, a, \delta_0(x_0)$  and initial data,   which
 may also vary from line to line.  We write $f(s)\sim g(s)$ to indicate 
$\displaystyle{\lim_{s\to \infty}\frac{f(s)}{g(s)}=1}.$

\section{A polynomial bound for the
$H^1\times L^2(B)$
  norm of   solution of equation  \eqref{A}} \label{section2}
%

Let us first recall the rough polynomial  space-time estimate of the solution $u$   of \eqref{gen} near any
non characteristic point  obtained in \cite{HZjmaa2020}
(see Theorem 1).  
 More precisely, we established the
following results:

\medskip

\noindent {\it{\bf { (Polynomial space-time estimate of solution of
\eqref{A}).}}}
{\it
\noindent  Consider $u $   a solution of \eqref{gen} with
blow-up graph $\Gamma:\{x\mapsto T(x)\}$ and  $x_0$  a non
characteristic point. Then,  there exists
$t_{0}(x_{0})\in [ 0,T(x_{0}))$ and  $q=q(a,p,N)>0$  such that,   for all $T_0 \in (t_{0}(x_0),T(x_{0})]$,  for all $s\geq -\log (T_0-t_{0}(x_{0}))$ and $x\in \er^N$ where $|x-x_0|\le \frac{T_0-t}{\delta_0(x_0)}$, 
we have
\begin{equation}\label{feb19}\int_s^{s+1}\ibint \big( |\grad w|^2 +( \partial_{s}w)^2\big)\y \t \leq K_1 s^q,
\end{equation}
\begin{equation}\label{feb191}\frac1{s^a}\int_s^{s+1}\ibint | w|^{p+1}\log^a(2+\phi^2  w^2)\y \t \leq K_1 s^q,
\end{equation}
\begin{equation}\label{feb192}
\ibint  w^2\y \leq K_1 s^q,
\end{equation}
where $w=w_{x,T^*(x)}$ is defined in \eqref{scaling}, with
 \begin{equation}\label{18dec1}
T^*(x)=T_0-\delta_0(x_0)(x-x_0)
\end{equation}
 and $\delta_{0}(x_{0})$ defined in \eqref{nonchar}.  Note that  $K_1$ depends on $ p, a, N,  \delta_{0}(x_{0})$, 
 $T(x_{0})$, $t_0(x_0)$ and
$\|(u(t_0(x_0)),\partial_tu(t_0(x_0)))\|_{H^{1}\times
L^{2}(B(x_0,\frac{T(x_0)-t_0(x_0)}{\delta_0(x_0)}) )}$.
Moreover, we have
\begin{equation}\label{23fev1}
-K_1s^{q}  \leq H_{m_0}(w(s),s) \leq K_1s^{q},
\end{equation}
where \begin{equation}\label{F022}
H_{m_0}(w(s),s)=E(w(s),s)-\frac{m_0}{s} \ibint   w\partial_{s}w\w \y,
\end{equation}
$E(w(s),s)$ is given by \eqref{E} and $m_0$   is a sufficiently large constant.
}

\bigskip

This section is devoted to deriving  
a polynomial bound for the
$H^1(B)$
norm.
  More precisely, this is the aim of this section.

\begin{prop}\label{prop3.1}
\noindent  Consider $u $   a solution of ({\ref{gen}}) with
blow-up graph $\Gamma:\{x\mapsto T(x)\}$ and  $x_0$  a non
characteristic point. Then,  there exists
$t_{0}(x_{0})\in [ 0,T(x_{0}))$ and  $q_1=q_1(a,p,N)>0$  such that  for all $T_0 \in (t_{0}(x_0),T(x_{0})]$, $s\geq -\log (T_0-t_{0}(x_{0}))$
and $x\in \er^N$ with $|x-x_0|\le \frac{T_0-t}{\delta_0(x_0)}$, 
we have
\begin{equation}\label{co2bis}
\|w(y,s)\|_{H^1(B)}
\leq K_2 s^{q_1},
\end{equation}
where $w=w_{x,T^*(x)}$ is defined in \eqref{scaling},  with $T^*(x)$ given in 
\eqref{18dec1}
 and $\delta_{0}(x_{0})$ defined in \eqref{nonchar}.  Note that  $K_2$ depends on $ p, a,  \delta_{0}(x_{0})$, 
 $T(x_{0})$, $t_0(x_0)$ and \\
$\|(u(t_0(x_0)),\partial_tu(t_0(x_0)))\|_{H^{1}\times
L^{2}(B(x_0,\frac{T(x_0)-t_0(x_0)}{\delta_0(x_0)}) )}$.
\end{prop}

\begin{nb}
  Let us insist on the fact that
the strategy of the proof works  only in   the subconformal case. 
Obviously, we can also prove that $\|\partial_sw(y,s)\|_{L^2(B)}\leq K_2 s^{q_1}.$
However, since this estimate is not useful in the proof, we have
chosen  not to  include it in the above proposition.
\end{nb}
\begin{nb}\label{polybound}
By using  the Sobolev's embedding and the above proposition, we can
deduce  that  for all $r\in [2,2^*]$ where $2^*=\frac{2N}{N-2}$ if
$N\ge 3$,
and for all $r\in [2,\infty)$
if $N=2$:
\begin{equation}\label{16dec5bis}
\|w(s)\|_{L^{r}(B)} \le K_3s^{q_1},\quad  \textrm{for all }\ \   s\geq -\log (T^*(x)-t_{0}(x_{0})),
\end{equation}
where $K_3$ depends also on $r$.


\end{nb}

Let us prove Proposition \ref{prop3.1} in the following.

\medskip

  {\it{Proof of Proposition \ref{prop3.1}:}}
We proceed in 2 steps:\\
- In Step 1, we use   the covering technique  and  interpolation  to derive a polynomial estimate related
to the $L^{\frac{p+3-\varepsilon}{2}}(B)$ norm  of   $w(s)$, for any $\varepsilon\in (0,p-1).$\\
- In Step 2,
using Step 1,
a Gagliardo-Nirenberg estimate and estimate \eqref{23fev1}  satisfied by $H_{m_0}(w(s),s)$ defined in \eqref{F022},
we easily conclude   the proof of estimate \eqref{co2bis}.

\bigskip

{\bf Step 1:  Control of  $w$ in $L^{\frac{p+3-\varepsilon}{2}}(B)$}

We claim the following:
\begin{lem}\label{lem22}  For all $\varepsilon \in (0,p-1)$ and
  $s\ge  -\log (T^*(x)-t_0(x_0)),$  we have 
\begin{equation}\label{s1jan19}
\iint\!|w(y,s )|^{\ppp}{\mathrm{d}}y
\le K_4(\varepsilon) s^{q}.
\end{equation}
\end{lem}
\noindent{\it Proof}:
Let us recall from the expression of $\p= \ps$ defined in \eqref{defphi}  that  we have for all $s\ge \max  (-\log T^*(x),1)$, 
\begin{equation}\label{id1}
e^{-\frac{2(p+1)s}{p-1}}s^{\frac{2a}{p-1}}  \p wf(\p w)=\frac1{s^{a}}|w|^{p+1}\log^a(2+\p^2w^2).  
\end{equation}
Combining  \eqref{id1}, \eqref{equiv1}, \eqref{equiv4bis} and  \eqref{feb191}, we deduce that for all $s\ge  -\log (T^*(x)-t_0(x_0)),$ 
\begin{equation}\label{13fev2}\int_s^{s+1}\ibint | w|^{p+1-\varepsilon}\y \t \leq C (\varepsilon)
 K_1 s^q.
\end{equation}
 Now,  for all $s\ge  -\log (T^*(x)-t_0(x_0))$, 
  using the mean value theorem, we derive the existence of $\sigma(s)\in [s,s+1]$ such that
\begin{equation}\label{s1}
\iint\!|w(y,\sigma (s))|^{\pn}{\mathrm{d}}y
=\int_{s}^{s+1}\iint\!|w(y,\tau
)|^{\pn}{\mathrm{d}}y{\mathrm{d}}\tau.
\end{equation}
 Let us introduce the following identity for all $s\ge  -\log (T^*(x)-t_0(x_0)),$  
\begin{align}\label{2018d1}
\iint\!|w(y,s )|^{\ppp}{\mathrm{d}}y=&\iint\!|w(y,\sigma (s))|^{\ppp}{\mathrm{d}}y  +
\int_{\sigma (s)}^{s}\frac{d}{d \tau}\iint |w(y,\tau)|^{\ppp}{\mathrm{d}}y{\mathrm{d}}\tau.
\end{align}
Combining (\ref{s1}), \eqref{2018d1} and   the fact that  $xy\le x^2+y^{2},$ for all $x\ge 0,$  $y\ge 0$, we write  for all $s\ge  -\log (T^*(x)-t_0(x_0)),$ 
\begin{align}\label{janv121}
\iint\!|w(y,s )|^{\ppp}{\mathrm{d}}y
\le& \int_{s}^{s+1}\iint\!|w(y,\tau
)|^{\ppp}{\mathrm{d}}y{\mathrm{d}}\tau
 + C\int_{s}^{s+1}\iint\!|w(y,\tau)|^{\pn}{\mathrm{d}}y{\mathrm{d}}\tau\no\\
&+C\int_{s}^{s+1}\iint\!(\partial_s
w(y,\tau))^2{\mathrm{d}}y{\mathrm{d}}\tau.
\end{align}
Thanks to   (\ref{janv121}), the classical inequality $x^{\ppp}\le 1+x^{\pn},$ for all $x\ge 0$,  $\varepsilon \in (0,p-1)$,   (\ref{feb19}) and  (\ref{13fev2}),  we obtain
\eqref{s1jan19}.
This concludes  the Lemma \ref{lem22}. 
\Box

\bigskip

{\bf Step 2:   Control of  $\grad w$ in $L^{2}(B)$}.

As in the pure power case, we first use the  Gagliardo-Nirenberg inequality in order to obtain   the following:
\begin{lem}\label{lem24}  There exists $\varepsilon_0=\varepsilon_0(p,N)>0$ such that, for all $\varepsilon \in (0,\varepsilon_0]$, for all   
  $s\ge  -\log (T^*(x)-t_0(x_0)),$  we have 
\begin{equation}\label{s40}
\iint\!| w(y,s )|^{p+1+\varepsilon}{\mathrm{d}}y
\le  K_5 s^{q} \Big(\iint\!|\grad w(y,s )|^{2}{\mathrm{d}}y\Big)^{\beta(\varepsilon) }+K_5s^{2q},
\end{equation}
where $\beta=\beta(p,N,\varepsilon)
 \in (0,1).$
\end{lem}
\noindent{\it Proof}:
We distinguish two cases:
\begin{itemize}
\item First case ($N=2$):\\
Let $\varepsilon>0$ and  $r=r(p,\varepsilon)>p+1+\varepsilon$.   By interpolation, we write 
\begin{equation}\label{N2v1}
\iint\!| w(y,s )|^{p+1+\varepsilon}{\mathrm{d}}y
\le  \Big(\iint\! |w(y,s)|^{\ppp}{\mathrm{d}}y\Big)^{\eta} \Big(\iint| w(y,s )|^{r}{\mathrm{d}}y\Big)^{1-\eta},
\end{equation}
where
$$\eta=\frac{r-{(p+1+\varepsilon)}}{r-\frac{p+3-\varepsilon}{2}}.$$
By using   \eqref{N2v1} 
and  the  Sobolev 
embedding  
 $H^1(B)\hookrightarrow L^{q}(B)  $  we get
\begin{equation}\label{s39bis}
\iint\!| w(y,s )|^{p+1+\varepsilon}{\mathrm{d}}y
\le  \Big(\iint\! w^{\ppp}(y,s){\mathrm{d}}y\Big)^{\eta} \Big(\iint| w(y,s )|^{2}{\mathrm{d}}y+
\iint|\grad w(y,s )|^{2}{\mathrm{d}}y\Big)^{\beta }.
\end{equation}
where
$$
\beta(r,p,\varepsilon)=\frac{r(p-1+3\varepsilon)}{4r-2(p+3-\varepsilon)}.$$
The combination of   $\ds \lim_{r\to \infty}\beta(r,p,\varepsilon)=\frac{p-1+3\varepsilon}{4}$ and $\frac{p-1}4<1$ implies  that 
there  exists   $\varepsilon_0=\varepsilon_0(p)=\frac{5-p}{10}>0$ 
    such that for $\varepsilon \in (0,\varepsilon_0] $ we can choose $r$ large enough such that 
$\beta=\beta(r,p,\varepsilon) \in (0,1).$ 
The result \eqref{s40}  follows  immediately from  \eqref{s39bis},
\eqref{feb192}
and \eqref{s1jan19}.

\item Second case ($N\ge3$):\\
Let $\varepsilon \in (0,p-1)$.  Let us write
the following Gagliardo-Nirenberg inequality:
\begin{equation}\label{s39}
\iint\!| w(y,s )|^{p+1+\varepsilon}{\mathrm{d}}y
\le  \Big(\iint\! w^{\ppp}(y,s){\mathrm{d}}y\Big)^{\eta} \Big(\iint| w(y,s )|^{2}{\mathrm{d}}y+
\iint|\grad w(y,s )|^{2}{\mathrm{d}}y\Big)^{\beta},
\end{equation}
where 
$$\eta=\frac{1-\frac{p+1+\varepsilon}{2^*}}{1-\frac{p+3-\varepsilon}{2\cdot
    2^*}}, \ \ 
\beta=\beta(p,N,\varepsilon)=\frac{\frac{p-1+3\varepsilon}4}{1-\frac{p+3-\varepsilon}{2\cdot 2^*}} \quad {\textrm {and}} \quad  \frac{1}{2^*}=\frac{N-2}{2N}.
$$
Observe that the function  $p \hookrightarrow  \beta (p,N,\varepsilon)$
is an increasing function on $(1,p_c)$, hence, 
\begin{equation}\label{ABC}
   \beta (p,N,\varepsilon)<\beta ( p_c,N,\varepsilon ),   \qquad  \forall  p\in (1,p_c).
\end{equation} 
Thanks to  \eqref{ABC}, the fact $\beta ( p_c,N,0)=1$
and by continuity, we infer that  there exists $\varepsilon_0=\varepsilon_0(p,N)>0$  small enough 
such that for $\varepsilon \in (0,\varepsilon_0] $
 we have  $\beta=\beta(p,N,\varepsilon) \in (0,1).$ 
The result \eqref{s40}  follows  immediately from  \eqref{s39},
\eqref{feb192}
and \eqref{s1jan19},
which ends the proof of Lemma \ref{lem24}.
\Box
\end{itemize}
%


Now, we are ready to give the proof of Proposition \ref{prop3.1}.

\medskip

{{\it {Proof of Proposition \ref{prop3.1}}}: Let us first use the
  following covering lemma from \cite{MZimrn05}: using 
  Proposition 3.3 in
that paper
  with $\eta= \frac{2(p+1)}{p-1}-N$ (which is positive), $q=2$ and
  $f=\nabla u$, together with the self-similar change of variables \eqref{scaling}, we write
\begin{equation}
\d\sup_{\{x\;|\;|x-x_0|\le \frac{T_0-t}{\delta_0}\}}\int_B\left|\nabla w\right|^2 \y
\le C(\delta_0) \d\sup_{\{x\;|\;|x-x_0|\le \frac{T_0-t}{\delta_0}\}}\int_{B(0,{\frac12})}\left|\nabla w\right|^2  \y, \label{first}
\end{equation}
where $T^*(x)$ is given by \eqref{18dec1}, $s= -\log(T^*(x)-t)$  and
$w=w_{x,T^*(x)}\left(y,s \right).$
From  \eqref{23fev1}, the definition \eqref{F022} of   $H_{m_0}(w(s),s)$, we see that for all  $s\ge  -\log (T^*(x)-t_0(x_0)),$
\begin{align}\label{19fev1bisb}
\iint |\nabla w ( y,s)|^2(1-|y|^2)\w \y +\iint (\partial_sw ( y,s))^2w \y 
-\frac{2m_0}{s} \ibint   w\partial_{s}w\w \y\nonumber\\
  \le 2\iint
e^{-\frac{2(p+1)s}{p-1}}s^{\frac{2a}{p-1}}   F(\p w)\w \y + 2K_1s^{q}.
\end{align}
By the use of the basic  inequality $2ab\le a^2+b^2$, we write
\begin{align}\label{19fev1bis12}
\frac{2m_0}{s} \ibint   w\partial_{s}w\w \y
  \le
\iint (\partial_sw ( y,s))^2w \y 
+\frac{(m_0)^2}{s^2} \ibint  w^2\w \y.
\end{align}
Plugging  \eqref{19fev1bis12} and  \eqref{feb192}
 into \eqref{19fev1bisb}, we obtain
\begin{align}\label{19fev1bis}
\iint |\nabla w ( y,s)|^2(1-|y|^2)\w \y  &\le 2\iint
e^{-\frac{2(p+1)s}{p-1}}s^{\frac{2a}{p-1}}   F(\p w)\w \y + K_6s^{q}.
\end{align}
Thanks  to  \eqref{equiv4} and  \eqref{19fev1bis}, we conclude for all  $s\ge  -\log (T^*(x)-t_0(x_0)),$
\begin{align}\label{19fev1}
\iint |\nabla w(y,s )|^2(1-|y|^2)\w \y  
 &\le K_7\iint 
|w(y,s )|^{p+ \varepsilon+1}\w \y + K_7s^{q}.
\end{align}
According to \eqref{19fev1} together with  Lemma \ref{lem24}, we have for all $s\ge  -\log (T^*(x)-t_0(x_0)),$
\begin{align}\label{19fev2}
\d\int_{B(0,\frac12)}|\nabla w(y,s )|^2\y  &\le C
\iint |\nabla w(y,s )|^2(1-|y|^2)\w \y  \nonumber\\
&\le  K_8s^{q} \Big(\int\!|\grad w(y,s )|^{2}{\mathrm{d}}y\Big)^{\beta}+K_8s^{2q}.
\end{align}
Therefore,
\begin{equation}
\d\sup_{\{x\;|\;|x-x_0|\le \frac{T_0-t}{\delta_0}\}}\int_{B(0,\frac12)}|\nabla w(y,s )|^2\y
\le K_8s^{q} \left(\d\sup_{|x-x_0|\le \frac{T_0-t}{\delta_0}}\int_B\left|\nabla w(y,s )\right|^2  \y\right)^\beta+K_8s^{2q}.\label{second}
\end{equation}
where $\beta\in (0, 1)$. From \eqref{first} and \eqref{second}, we see that
\begin{equation}
\d\sup_{\{x\;|\;|x-x_0|\le \frac{T_0-t}{\delta_0}\}}\iint|\nabla w(y,s )|^2\y
\le K_9s^{q} \left(\d\sup_{|x-x_0|\le \frac{T_0-t}{\delta_0}}\int_B\left|\nabla w(y,s )\right|^2  \y\right)^\beta+K_9s^{2q}.\label{second22}
\end{equation}
It suffices to combine \eqref{second22}  and the fact that $\beta<1$, to obtain that
\begin{equation}
\d\sup_{\{x\;|\;|x-x_0|\le \frac{T_0-t}{\delta_0}\}}\iint|\nabla w(y,s )|^2\y
\le K_{10}s^{\frac{2q}{1-\beta}}.\label{second222}
\end{equation}
Clearly, by using  \eqref{second222} and \eqref{feb192}, we  conclude \eqref{co2bis}, 
where $q_1=\frac{1}{1-\beta}q$, which yields the conclusion of Proposition
\ref{prop3.1}. \Box


%

\bigskip

\section{Proof of Theorem \ref{t1} and Theorem \ref{t2}}\label{section3}

In this section, we 
 prove 
Theorem \ref{t1} and Theorem \ref{t2} here  thanks to   Proposition \ref{prop3.1}. 
This section is divided into two parts:
\begin{itemize}
\item  In subsection \ref{3.2},   we state a general version of Theorem \ref{t1}, uniform for $ x$ near $x_0$ and prove it.
%
\item In subsection \ref{3.3}, we prove   Theorem \ref{t2}.
\end{itemize}

\subsection{A Lyapunov functional}\label{3.2}
In this subsection, our aim is to construct a Lyapunov functional for equation \eqref{A}.
 Note that this functional is far from being trivial and makes our main contribution.
More precisely, thanks to the  rough estimate obtained in  the Proposition \ref{prop3.1}, 
we derive here that the functional  $L(w(s),s)$ defined in \eqref{10dec2} is a decreasing 
  functional  of time  for equation (\ref{A}),  provided that $s$ is large enough. 
First,  thanks to 
the  additional information obtained in Section \ref{section2}, we can write this
 important lemma which plays a key role in our analysis. More precisely, we claim the following:
\begin{lem}  \label{lemmain}  For all   
  $s \geq -\log (T^*(x)-t_0(x_0))$, we have
\begin{align}\label{claim1}
\int_{B}  {|  w|^{p+1}}\log^{{a}}(2+\p^2 w^2  )\log (2+w^2 ) \w \y\le&
 K_{11}{s^{\frac14}}\int_{B}  {|  w|^{p+1}}\log^{{a}}(2+\p^2 w^2  ) \w \y\nonumber\\
&+ K_{11} {s^{a+\frac14}}.
\end{align}
\end{lem}
\begin{nb} Let us mention that, in  the first term on the right-hand side, the choice of the 
 power $\frac14$  is not optimal. In fact, with the same proof, one can show  the same estimate
with the  power   $\nu$, for any $\nu>0$, instead of the power $\frac14$. Let us remark that  we can construct a Lyapunov functional, when we have the estimate above for some  power $\nu$  such that   $\nu\in (0,1 )$  instead of the power $\frac14$.
\end{nb}
{\it Proof:} Let $\varepsilon \in (0,1)$. By
using  the inequality   $\log (2
+z^2)\le C(\varepsilon)+|z|^{\varepsilon^2}$,    for all $ z\in \R$, we conclude that
\begin{align}\label{22fev1}
\int_{B}  \! {|  w|^{p+1}}\log^{{a}}(2+\p^2 w^2  )\log (2+w^2 ) \w \y\le \! C \!\int_{B}  \! {|w|^{p+1}}\log^{{a}}(2+\p^2 w^2  ) \w \y\nonumber\\
+\int_{B}  {|  w|^{p+1+\varepsilon^2}}\log^{{a}}(2+\p^2 w^2  ) \w \y.\qquad \qquad
\end{align}
Furthermore, we apply the interpolation in Lebesgue spaces to
get
\begin{align}\label{22fev2}
\int_{B}  {|  w|^{p+1+\varepsilon^2}}\log^{{a}}(2+\p^2 w^2  ) \w \y\le\Big( \int_{B}  {|  w|^{p+1}}\log^{{a}}(2+\p^2 w^2  ) \w \y\Big)^{1-\varepsilon}\nonumber\\
\Big( \int_{B}  {|  w|^{p+1+\varepsilon}}\log^{{a}}(2+\p^2 w^2  ) \w \y\Big)^{\varepsilon}.
\end{align}
By combining\eqref{equiv1}, \eqref{equiv4} and   the inequality   $|z|^{\varepsilon}\le 1+|z|^{p+1+2\varepsilon}$,    
for all $ z\in \R$, we obtain
\begin{align}\label{22fev22}
\frac1{s^{ a}}\int_{B}  {|  w|^{p+1+\varepsilon}}\log^{{a}}(2+\p^2 w^2  ) \w \y\le
 C
+C \int_{B}  {|  w|^{p+1+2 \varepsilon}}  \y.
\end{align}
Since $ p<p_c=\frac{N+3}{N-1}$, we then  choose $ \varepsilon_1\le \varepsilon_0$ small enough, such that for all $\varepsilon \in (0,\varepsilon_1 ]$  we have $p+1+2 \varepsilon < 2^* $
where $2^*=\frac{2N}{N-2}$, if $N\ge 3$ and $2^*=\infty$,  if $N=2$.
Therefore,  estimate   \eqref{16dec5bis}  implies that, for all  $ s\geq -\log (T^*(x)-t_{0}(x_{0}))$
\begin{align}\label{22fev555}
\int_{B}  {| w|^{p+1+2 \varepsilon}}  \y\le  
(K_3 s^{q_1})
^{p+1+2\varepsilon},\quad \forall  \varepsilon\in [0,\varepsilon_1].
\end{align}
By  combining \eqref{22fev2}, \eqref{22fev22} and   \eqref{22fev555}, we deduce that, for all  $ s\geq -\log (T^*(x)-t_{0}(x_{0}))$,  for all $ \varepsilon\in ( 0,\varepsilon_1]$.
\begin{align}\label{22fev222}
\int_{B}  {|  w|^{p+1+\varepsilon^2}}\log^{{a}}(2+\p^2 w^2  ) \w \y\le K_{12}
 s^{q_1(p+1+2\varepsilon)\varepsilon}s^{\varepsilon a}
\Big( \int_{B}  {|  w|^{p+1}}\log^{{a}}(2+\p^2 w^2  ) \w \y\Big)^{1-\varepsilon}.
\end{align}
Thanks to  the basic inequality  $|X|^{\nu}|Y|^{1-\nu}\le C|X|+C |Y|$,  for all $ X,Y\in \R$, for all $ \nu \in   ( 0, 1 ), $ we conclude that, for all  $ s\geq -\log (T^*(x)-t_{0}(x_{0}))$,  for all $ \varepsilon\in ( 0,\varepsilon_1]$.
\begin{align}\label{22fev222bis}
\int_{B}  {|  w|^{p+1+\varepsilon^2}}\log^{{a}}(2+\p^2 w^2  ) \w \y
\le K_{13} s^{q_1(p+1+2\varepsilon)\varepsilon}
\Big(s^{ a}+ \int_{B}  {|  w|^{p+1}}\log^{{a}}(2+\p^2 w^2  ) \w \y\Big).
\end{align}
We choose  $ \varepsilon_2\in   ( 0,\varepsilon_1], $ such that $  q_1(p+1+2\varepsilon_2)\varepsilon_2<\frac14.$
Then,  by \eqref{22fev1} and  \eqref{22fev222bis}, 
we
easily obtain \eqref{claim1}. This concludes the proof  of Lemma \ref{lemmain}.
\Box

\bigskip

Thanks to estimate \eqref{claim1},  we can improve the  estimate related to the 
 control of  the time derivative of the  functional $E(w(s),s)$. More precisely, we prove  the following lemma:
\begin{lem}\label{2018lem31} There exists  $S_1>0 $ such that   for all   
  $s \geq \max ( -\log (T^*(x)-t_0(x_0)), S_1)$, we have  
\begin{align}\label{E01}
\frac{d}{ds}E(w(s),s)\le &-  \frac{3\alpha}{2}\iint (\partial_{s}w)^2\frac{\w}{1-|y|^2}\y
+ 
 \frac{K_{14}}{s^{a+\frac74}}\iint |w|^{p+1}\log^a(2+\p^2w^2)\w \y\no\\
&+\frac{C}{s^2}\ibint  |\grad w|^2(1-|y|^2)\w \y+\frac{C}{s^2}\ibint  w^2\w \y   +\frac{K_{14}}{s^{\frac74}}.
 \end{align}
\end{lem}
%
%
{\it Proof}: Multiplying $\eqref{A}$ by $\partial_{s} w\!\ \w$ and integrating over  $B$, we obtain
\begin{align}
\frac{d}{ds}E(w(s),s)=& - 2\alpha \ibint (\partial_{s}w)^2\frac{\w}{1-|y|^2}\y\label{E00}\\
&+\underbrace{\frac{a}{(p+1)s^{a+1}}\iint  {|  w|^{p+1}}\log^{{a-1}}(2+\p^2 w^2  )\Big(\log (2+\p^2w^2 ) -\frac{4s}{p-1}\Big)\w \y
}_{\chi_{1}(s)}\no\\
&+\underbrace{\frac{2e^{-\frac{2(p+1)s}{p-1}}}{p-1}s^{\frac{2a}{p-1}}\iint \Big((p+1) F_2(\p w)-\frac{a}{s} F_1(\p w)-\frac{a}{s}F_2(\p w)\Big)\w \y}_{\chi_ 2(s)}\no\\
&+\underbrace{\gamma (s)\ibint w\partial_sw\w\y
+\frac{2a}{(p-1)s}\ibint (\partial_{s}w)^2\w\y}_{\chi_3(s)}\no\\
&+\underbrace{
\frac{2a}{(p-1)s}\ibint  y.\grad w\partial_{s}w\w \y
}_{\chi_4(s)},\no
 \end{align}
where
$F_1$ and $F_2$ are defined by
\begin{equation}
F_1(x)= -\frac{ 2a} {(p+1)^2}{| x|^{p+1}}\log^{{a-1}}(2+x^2  ),\label{defF2}
\end{equation}
and 
\begin{equation}\label{defF123}
F_2(x)=F(x)-\frac{xf(x)}{p+1}-F_1(x).
\end{equation}
Note that, in \eqref{E00}
 we  grouped  the main terms together.  In fact, it is easy to  control the terms  $\chi_{2}(s)$, $\chi_{3}(s)$  and  $ \chi_{4}(s)$. However, the  control of  the  term $\chi_{1}(s)$
  needs the  use of  the  additional information obtained in Lemma \ref{lemmain}.
 More precisely,
%
for all   
  $s \geq -\log (T^*(x)-t_0(x_0))$, we   divide $B$ into two parts
 \begin{equation}\label{27nov1}
A_{1}(s)=\{y \in B\,\,|\,\, \ps w^2(y,s)\leq  1\}\,\,{\rm and }\,\,A_{2}(s)=\{y \in B
\,\,|\,\, \ps w^2(y,s)\ge  1\}.
\end{equation}
Accordingly, we write  $\chi_1(s)=\chi_1^1(s)+\chi_1^2(s)$, where
\begin{align}
\chi_1^{1}(s)=&
\frac{a}{(p+1)s^{a+1}}\int_{A_1(s)}  {|  w|^{p+1}}\log^{{a-1}}(2+\p^2 w^2  )\Big(\log (2+\p^2w^2 ) -\frac{4s}{p-1}\Big)\w 
,
\nonumber\\
\chi_1^{2}(s)=&
\frac{a}{(p+1)s^{a+1}}\int_{A_2(s)}  {|  w|^{p+1}}\log^{{a-1}}(2+\p^2 w^2  )\Big(\log (2+\p^2w^2 ) -\frac{4s}{p-1}\Big)\w \y.
\nonumber
\end{align}
On the one hand, by using 
 the definition of the set $A_1(s)$ given  in \eqref{27nov1} and   the expression of $\ps$ in  \eqref{defphi},   we get, for all   $s \geq -\log (T^*(x)-t_0(x_0))$,
\begin{equation}\label{16dec1}
 | w|^{p+1}\log^{{a}}(2+\p^2 w^2  )\le  C\p^{-\frac{p+1}2}(s)\log^{|a|}(2+\ps) \le Ce^{-\frac{ps}{p-1}}.
\end{equation}
If we integrate \eqref{16dec1} over $A_1(s)$, we obtain
\begin{equation}\label{93}
\chi^1_1(s) \le Ce^{-s}.
\end{equation}
On the other hand, by using  the definition of the $\ps$ given by \eqref{defphi}, we write the identity
\begin{equation}\label{16dec102}
\log (2+\p^2 w^2)
 -\frac{4s}{p-1}=\log (2\p^{-2}+ w^2)-\frac{2a\log s}{p-1}.
\end{equation}
Furthermore, the exists   $S_0>0 $ such that   for all    
  $s \geq  S_0$, we have   $\phi  (s)\ge 1 $. Therefore,
 by exploting  \eqref{16dec102}, we write  for all 
 $s \geq \max ( -\log (T^*(x)-t_0(x_0)), S_0)$, 
\begin{equation}\label{16dec10}
\log (2+\p^2 w^2)
 -\frac{4s}{p-1}\le  \log (2+w^2)+C\log s.
\end{equation}
Also, by using the definition of the set $A_2(s)$ defined in \eqref{27nov1}, we can write 
 for all $ s \geq  -\log (T^*(x)-t_0(x_0)),$   if $y\in A_{2}(s)$, we have 
\begin{equation}\label{16dec12}
\log(2+\p^2w^2)\ge \log(\ps)\ge \frac{2s}{p-1}-\frac{a \log s}{p-1}.
\end{equation}
Clearly, the exists   $S_1>S_0 $ such that   for all    
  $s \geq  S_1$, we have   $  \frac{2s}{p-1}-\frac{a \log s}{p-1}\ge  \frac{s}{p-1}$. Therefore,
 by exploiting  \eqref{16dec10} and  \eqref{16dec12} we have   for all  $s \geq \max ( -\log (T^*(x)-t_0(x_0)), S_1)$, 
\begin{align}
\chi_1^{2}(s)\le &
\frac{C}{s^{a+2}}\int_{B}  {|  w|^{p+1}}\log^{{a}}(2+\p^2 w^2  )\log (2+w^2 ) \w \y\nonumber\\
&+\frac{C\log s}{s^{a+2}}\int_{B}  {|  w|^{p+1}}\log^{{a}}(2+\p^2 w^2  )\w \y.\label{131bis}
\end{align}
Adding  \eqref{claim1} and \eqref{131bis} we have for all  $s \geq \max ( -\log (T^*(x)-t_0(x_0)), S_1)$, 
\begin{equation}\label{16dec14}
\chi_1^{2}(s)\le 
\frac{K_{15}}{s^{a+\frac74}}\iint   {|  w|^{p+1}}\log^{{a}}(2+\p^2 w^2 )\w \y+\frac{K_{15}}{s^{\frac74}}.
\end{equation}
Note that, by using  the fact $\chi_1(s)=\chi_1^{1}(s)+\chi_1^{2}(s)$, \eqref{93} and \eqref{16dec14}, we get
\begin{equation}\label{13janva1}
\chi_1(s)\le 
\frac{K_{16}}{s^{a+\frac74}}\iint   {|  w|^{p+1}}\log^{{a}}(2+\p^2 w^2 )\w \y+\frac{K_{16}}{s^{\frac74}}.
\end{equation}
Note from    \eqref{equiv2}  and  \eqref{equiv3}  that
\begin{equation}\label{15dec1}
\frac{1}{s}| F_1(\p w)|+| F_2(\p w)|\le   C+C \frac{ \p w}{s^2}f(\p w).
\end{equation}
By \eqref{E00}, \eqref{15dec1} and \eqref{id1},  we have, for all $s\ge  -\log (T^*(x)-t_0(x_0))$,
\begin{equation}\label{sigma11dec18}
\chi_2(s)\le 
\frac{C}{s^{a+2}}\iint |w|^{p+1}\log^a(2+\p^2w^2)\w \y+  C e^{-2s}.
\end{equation}
Finally,  by using  the following  basic inequality
\begin{equation}\label{basic1}
ab\le \nu a^2+\frac1{\nu}b^2, \ \forall \nu >0,
\end{equation}
and the expression of $\gamma(s)$ defined in 
  \eqref{defgamma},  we  write, for all $s\geq -\log T^*(x)-t_0(x_0)$ 
\begin{equation}\label{sigma13}
\chi_{3}(s)+
\chi_{4}(s)\leq \frac\alpha2\ibint (\partial_{s}w)^2\frac{\w}{1-|y|^2}\y+ \frac{C}{s^2}\ibint  \Big(|\grad w|^2(1-|y|^2)+w^2\Big)\w \y.
\end{equation}
The result \eqref{E01} derives immediately from  \eqref{E00}, \eqref{sigma13},  \eqref{13janva1}, \eqref{sigma11dec18},  and  the identity \eqref{E00},
which ends the proof of Lemma \ref{2018lem31}
\Box

\medskip

Let us now recall the following result from  \cite{HZjmaa2020}, where 
we  write  an estimate  on the  functional $J(w(s),s)$ defined by:
\begin{equation}\label{JJ}
J(w(s),s) = 
- \frac{1}{s} \ibint  w\partial_{s}w\w\y.
\end{equation}
 \begin{lem}\label{LemJ}
For  all $s \geq \max  (-\log T^*(x),1)$, we have 
\begin{align}\label{6nov20181}
\frac{d}{ds}J(w(s),s)\ \  \le&\ \ \frac{p+3}{2s}E(w(s),s)-\frac{p+7}{4s}\ibint   (\partial_{s}w)^2\w \y\\
&-\frac{p-1}{4s}\ibint ( |\grad w|^2-(y.\grad w)^2)\w\y
- \frac{p+1}{2(p-1)s} \ibint  w^2\w\y\no\\
&
-\frac{p-1}{2(p+1)s^{a+1}}\ibint |w|^{p+1}\log^a(2+\p^2 w^2)  
\w\y+\Sigma_2(s),\no
\end{align}
where  $\Sigma_2(s)$ satisfies
  \begin{align}\label{DEC1}
  \Sigma_{2}(s)\leq& \frac{C}{\sqrt{s}}
 \ibint (\partial_{s}w)^2 \frac{\w}{1-|y|^2}\y+\frac{C}{s\sqrt{s}}\ibint   |\grad w|^2(1-|y|^2)\w\y \\
&+\frac{C}{s\sqrt{s}}\ibint  w^2\w\y+\frac{C}{s^{a+2}}\ibint |w|^{p+1}\log^a(2+\p^2w^2)\w \y+  C e^{-2s}.\no
\end{align}
\end{lem}
 {\it Proof}: See Lemma 2.2  in    \cite{HZjmaa2020}.

%
%
%
%
 \Box

\medskip

With Lemmas  \ref{2018lem31} and  \ref{LemJ},  we are in a position to state and prove
Theorem \ref{t1}', which is a uniform version of Theorem \ref{t1} for $x$ near $x_0$.


\bigskip

\noindent {\bf{Theorem} \ref{t1}'} {\it (Existence  of a Lyapunov functional for equation }
{\eqref{A}})\\
\label{t1bis}
{\it
Consider   $u $    a solution of ({\ref{gen}}) with blow-up graph
$\Gamma:\{x\mapsto T(x)\}$ and  $x_0$  a non characteristic point.
Then there exists $t_1(x_0)\in [0,T(x_0)) $ such that, 
 for all $T_0\in  (t_1(x_0),T(x_0)]$,  for all  $s\ge  -\log(T_0-t_1(x_0))$  and $x\in \er$, where $|x-x_0|\le \frac{T-t}{\delta_0(x_0)}$,
 we have
\begin{equation}\label{t1lyap}
 L(w(s+1),s+1)-L(w(s),s) \leq -\alpha \int_{s}^{s+1}\iint (\partial_{s}w)^2\frac{\w}{1-|y|^2}\y \t,
\end{equation}
where  $w=w_{x,T^*(x)}$ and $T^*(x)$  is defined in \eqref{18dec1}.
}

\medskip

  {\it{Proof of Theorem \ref{t1}':}}
By exploiting the defintion of $L_0(w(s),s)$ in \eqref{5jan1},
we  can write easily
\begin{equation}\label{14jan1}
\frac{d}{ds}L_0(w(s),s)=\frac{d}{ds}E(w(s),s) + \frac{1}{\sqrt{s}}\frac{d}{ds}J(w(s),s)-  \frac{1}{2s\sqrt{s}}J(w(s),s).
\end{equation}
With Lemmas  \ref{2018lem31} and  \ref{LemJ} and
the following inequality
$$
\frac{1}{2s^2\sqrt{s}}\iint   w\partial_{s}w\w \y+\frac{p+3}{2s^3}\iint   w\partial_{s}w\w \y\le 
\frac{C}{s^2}\iint   (\partial_{s}w)^2\w \y+\frac{C}{s^2}\iint   w^2\w \y,$$
 allows to prove that 
for all $s \geq \max  (-\log(T^*(x)-t_0(x_0)), S_{1})$, we have 
\begin{align*}
\frac{d}{ds}L_0(w(s),s)\le &-  (\frac{3\alpha}{2}- \frac{C}{s})\iint (\partial_{s}w)^2\frac{\w}{1-|y|^2}\y+ \frac{p+3}{2s\sqrt{s}}L_0(w(s),s)\no\\
&- \frac{1}{s\sqrt{s}}( \frac{p+1}{2(p-1)} -\frac{C}{\sqrt{s}})\iint  w^2\w \y  \\
&-\frac{1}{s\sqrt{s}}(\frac{p+7}{4}-\frac{C}{\sqrt{s}})\iint   (\partial_{s}w)^2\w \y\\
&-\frac1{s\sqrt{s}}(\frac{p-1}{4}-\frac{C}{\sqrt{s}})\iint  |\grad w|^2(1-|y|^2)\w\y
\no\\
&-\frac1{s^{a+\frac32}}(\frac{p-1}{2(p+1)}-  \frac{K_{14}}{s^{\frac14}}-\frac{C}{s})\iint |w|^{p+1}\log^a(2+\p^2w^2)  
\w\y\\
&+ C \frac{e^{-2s}}{\sqrt{s}}+\frac{K_{14}}{s^{\frac74}}.
  \end{align*}
Again,  choosing  $S_2> -\log(T(x_0)-t_0(x_0))$ large enough,  this  implies that  for all $s \geq \max  (-\log(T^*(x)-t_0(x_0)), S_{2})$, we have 
\begin{equation}\label{17dec1}
\frac{d}{ds}L_0(w(s),s)\le -\alpha \iint (\partial_{s}w)^2\frac{\w}{1-|y|^2}\y+ \frac{p+3}{2s\sqrt{s}}L_0(w(s),s)
+\frac{K_{15}}{s^{\frac74}}.
\end{equation}
Recalling that,
$$L(w(s),s)=\exp\Big(\frac{p+3}{\sqrt{s}}\Big) L_0(w(s),s)+\frac{\theta }{s^{\frac34}},$$  
   we get from straightforward computations 
  \begin{equation}\label{17dec2}
  \frac{d}{ds}L(w(s),s) =-\frac{p+3}{2s\sqrt{s}}\exp\Big(\frac{p+3}{\sqrt{s}}\Big) L_0(w(s),s)+\exp\Big(\frac{p+3}{\sqrt{s}}\Big)\frac{d}{ds}L_0(w(s),s)-\frac{3\theta}{4s^{\frac74}}.
\end{equation}
Therefore, estimates $\eqref{17dec1}$ and $\eqref{17dec2}$ lead to the following crucial estimate:
 \begin{equation}
 \frac{d}{ds}L(w(s),s)\leq -\alpha \exp\Big(\frac{p+3}{\sqrt{s}}\Big) \iint (\partial_{s}w)^2 \frac{\w}{1-|y|^2}\y
+\Big(K_{15}\exp\Big(\frac{p+3}{\sqrt{s}}\Big)-\frac{3\theta}4 \Big)\frac{ 1}{s^{\frac74}}.
\end{equation}
Since we have $1\leq \exp\Big(\frac{p+3}{\sqrt{s}}\Big)\leq \exp\Big(\frac{p+3}{\sqrt{S_2}}\Big)$, 
we then choose $\theta$ large enough, so that $
K_{15}\exp\Big(\frac{p+3}{\sqrt{s}}\Big)-\frac{3\theta}4 
\leq 0$, which yields,  for all $s \geq \max  (-\log(T^*(x)-t_0(x_0)), S_{2})$,
$$\frac{d}{ds}L(w(s),s)\leq -\alpha \iint (\partial_{s}w)^2 \frac{\w}{1-|y|^2}\y.$$
A simple integration between $s$ and $s+1$ ensures the result
\eqref{t1lyap}, where 
\begin{equation}\label{new19dec1}
 t_1(x_0)=\max(T(x_0)-e^{-S_2},t_0(x_0)).
\end{equation}
This concludes the proof
of  Theorem \ref{t1}'.
\Box

\medskip

 We now claim the following lemma:
\begin{lem}\label{L19}
 There exists $S_3\ge S_2$ such that, if
  $L(w(s_3),s_3)<0$ for some $s_3\ge \max(S_3,-\log (T^*(x)-t_1(x_0)))$, then $w$
blows up in some finite time $s_4>s_3$.
\end{lem}
 {\it Proof:}
The argument is the same  as  the similar part in  
\cite{HZjmaa2020}.

\Box

\subsection{Proof of Theorem \ref{t2} }\label{3.3}

In this subsection, we prove Theorem \ref{t2}.  Note that the lower bound follows from the finite speed of propagation and the wellposedness  in $H^1\times L^2$.
For a detailed argument in the similar case of equation (\ref{NLW}), see Lemma 3.1 (page 1136) in \cite{MZimrn05}.\\
\noindent We consider $u$  a solution of (\ref{gen}) which is defined under the graph of $x\mapsto T(x)$, and $x_0$ a non characteristic point.
Let 
\begin{equation}\label{21dec1}
 t_2(x_0)=\max(T(x_0)-e^{-S_3},t_1(x_0)).
\end{equation}
Given some $T_0\in (t_2(x_0),T(x_0)]$, for all
$x\in \er$ is such that $|x-x_0|\le \frac{T_0-t_2(x_0)}{\delta_0(x_0)}$, where $\delta(x_0)$ is defined in (\ref{nonchar}), we aim   at bounding $\|(w,\partial_s w)(s)\|_{H^1\times L^2(B)}$ for $s$ large.

\medskip


\noindent As in  \cite{ HZnonl12, omar1}, by combining Theorem  \ref{t1}' 
 and  Lemma \ref{L19}  we get the following bounds:
\begin{cor}\label{22jan1} { (Bound on  $L_0(w(s),s)$)}.
For all $T_0\in (t_2(x_0),T(x_0)]$, for all $s\ge  -\log (T_0-t_2(x_0))$
 and $x\in \er^N$ where $|x-x_0|\le \frac{T_0-t}{\delta_0(x_0)}$,
 we have
\begin{equation}\label{2018N00}
-C\leq L_{0}(w(s),s) \leq C L_{0}(w(\tilde{s}_{2}),\tilde{s}_{2})+C ,
\end{equation}
where $\tilde{s}_{2}=-\log (T^*(x)-t_2(x_0))$.\\
Moreover, 
 for all $s\ge  -\log (T^*(x)-t_2(x_0))$, we have
\begin{equation}\label{22janv2}
\int_{s}^{s+1}
 \iint (\partial_{s}w)^2 \frac{\w}{1-|y|^2}\y \s \le K_{16},
\end{equation}
where  $K_{16}=K_{16}(a,p,T^*(x),\|(u(t_2),u_t(t_2))\|_{
H^{1}\times L^{2}(B(x_0,\frac{T_0-t_2(x_0)}{\delta_0(x_0)})
)})$,
$C=C(a,p)$    and   $\delta_0(x_0)\in (0,1)$ is defined in (\ref{nonchar}).
\end{cor}
\medskip

\begin{nb} Using the definition of  (\ref{scaling}) of
$w_{x,T^*(x)}=w$, we write easily
\begin{equation}\label{2018N5}
   L_0(w(\widetilde{s_2}), \widetilde{s_2})
\le K_{17},
\end{equation}
where
${K_{17}}={K_{17}}(T(x_0)-t_2(x_0),\|(u(t_2(x_0)),\partial_tu(t_2(x_0)))\|_{H^{1}\times L^{2}(B(x_0,\frac{T(x_0)-t_2(x_0)}{\delta_0(x_0)}))})$.
\end{nb}

\medskip

\noindent Starting from these bounds, the proof of Theorem \ref{t2} is similar to the proof in \cite{MZajm03,MZimrn05}
except for
the treatment of the nonlinear terms and
of the perturbation terms. In our opinion, handling these terms is straightforward in
all the steps of the proof, except for the first step,  where we bound the time averages of the  nonlinear term and second step, where we remove the  time averages. However, the third step 
where we conclude the boundedness of the 
$H^{1}_{loc,u}(\er^N)$ norm of   solution of equation  \eqref{A} from Proposition \ref{projan22} is the same as in Proposition \ref{prop3.1} (up to some very minor changes).
  For that reason,
we only give the first  two step  and refer to  \cite{MZajm03,MZimrn05}  and the similar part in section \ref{section2} in this paper for the remaining steps in the proof of  Theorem \ref{t2}. This is the step we prove here.
\begin{prop}\label{projan22}
For all  $s\ge 1-\log (T^*(x)-t_3(x_0))$,
for some $t_3(x_0)\in [t_2(x_0), T(x_0))$, 
\begin{equation}\label{projan221}
\frac1{s^a}\int_{s}^{s+1}\!\!\iint  |w|^{p+1}\log^a(2+\p^2w^2)\w \y\t
\le K_{18}.
\end{equation}
\end{prop}
Proof:
For $s\ge 1-\log (T^*(x)-t_2(x_0))$, let us work with time integrals betwen $s_1$ et $s_2$ where $s_1\in [s-1,s]$
and $s_2\in [s+1,s+2]$.
 By integrating the expression (\ref{5jan1}) of $L_0(w(s),s)$ in time between $s_1$ and $s_2$, where $s_2>s_1>-\log (T^*(x)-t_2(x_0))$, we obtain:
\begin{align}\label{et}
\int_{s_1}^{s_2}\!\!L_0(w(s),s)\s=&\displaystyle\int_{s_1}^{s_2}\iint\!\!\Big ( \frac{1}{2}(\partial_{s}w)^2
+\frac{p+1}{(p-1)^2}w^2
-e^{-\frac{2(p+1)s}{p-1}}s^{\frac{2a}{p-1}}   F(\p w)\Big )\w {\mathrm{d}}y{\mathrm{d}}s\nonumber\\
&+\frac{1}{2}\displaystyle\int_{s_1}^{s_2}\!\!\iint\!\!  (|\grad w|^2-|y.\grad w|^2)\w {\mathrm{d}}y{\mathrm{d}}s-\int_{s_1}^{s_2}\!\! \frac1{s\sqrt{s}}\!\!\displaystyle\iint\!\!w\partial_s w\w {\mathrm{d}}y{\mathrm{d}}s.
\end{align}
By multiplying the equation (\ref{A}) by $w\w$ and integrating both in time and in space over $B\times [s_1,s_2]$  we obtain the following identity, after some
integration by parts:
%
\begin{eqnarray}\label{et1}
&&\Big [\iint\!\!\Big (w\partial_{s}w+(\frac{p+3}{2(p-1)}-N)w^2\Big ) \w{\mathrm{d}}y\Big ]_{s_1}^{s_2}=
\int_{s_1}^{s_2}\!\!\iint\!\!(\partial_{s}w)^2\w{\mathrm{d}}y{\mathrm{d}}s\\
&&-\int_{s_1}^{s_2}\!\!\iint\!\!\big(|\grad w|^2-(y.\grad w)^2\big)\w{\mathrm{d}}y{\mathrm{d}}s
-\frac{2p+2}{(p-1)^2}\int_{s_1}^{s_2}\!\!\iint\!\!w^2\w{\mathrm{d}}y{\mathrm{d}}s\nonumber\\
&&+\int_{s_1}^{s_2}\!\!\iint\!\!e^{-\frac{2ps}{p-1}}s^{\frac{a}{p-1}}w f(\p w)\w{\mathrm{d}}y{\mathrm{d}}s-2\alpha\!\int_{s_1}^{s_2}\!\!\iint\!\!w\partial_{s}w
\frac{|y|^2\w}{1-|y|^2}{\mathrm{d}}y{\mathrm{d}}s\nonumber\\
&&+2\!\!\int_{s_1}^{s_2}\!\!\iint\!\!y.\grad w\partial_{s} w \w\y{\mathrm{d}}s
+\frac{2a}{p-1}\int_{s_1}^{s_2}\iint \frac1{s}y.\grad w  w\w\y\s\nonumber\\
&&+\!\!\int_{s_1}^{s_2}\!\!\iint\!\!\gamma(s) w^2\w\y\s
+\frac{2a}{p-1}\int_{s_1}^{s_2}\iint \frac{1}{s}\partial_s w w\w \y\s.\nonumber
\end{eqnarray}
Note that, by using the identity  \eqref{defF123}, we get
\begin{align}\label{23jan12}
e^{-\frac{2(p+1)s}{p-1}}s^{\frac{2a}{p-1}}\Big(\frac{\p w}2 f(\p w)- F(\p w)\Big)&=\frac{p-1}{2}e^{-\frac{2(p+1)s}{p-1}}s^{\frac{2a}{p-1}} F(\p w)\\
&-\frac{p+1}2e^{-\frac{2(p+1)s}{p-1}}s^{\frac{2a}{p-1}}\Big(F_1(\p w)+F_2(\p w)\Big).\nonumber
\end{align}
By combining the identities (\ref{et}),  (\ref{et1}) and exploiting  \eqref{23jan12}, we obtain
\begin{eqnarray}\label{et44}
&&\frac{p-1}{2}\int_{s_1}^{s_2}\!\!\iint e^{-\frac{2(p+1)s}{p-1}}s^{\frac{2a}{p-1}} F(\p w)\w {\mathrm{d}}y{\mathrm{d}}s
\nonumber\\
&=&\frac12
\Big [\iint\!\!\Big (w\partial_{s}w+(\frac{p+3}{2(p-1)}-N)w^2\Big ) \w{\mathrm{d}}y\Big ]_{s_1}^{s_2}-
\int_{s_1}^{s_2}\!\!\iint\!\!(\partial_{s}w)^2\w{\mathrm{d}}y{\mathrm{d}}s\nonumber\\
&&+\int_{s_1}^{s_2}\!\!L_0(w(s),s)\s+\alpha\!\int_{s_1}^{s_2}\!\!\iint\!\!w\partial_{s}w
\frac{|y|^2\w}{1-|y|^2}{\mathrm{d}}y{\mathrm{d}}s\nonumber\\
&&-\!\!\int_{s_1}^{s_2}\!\!\iint\!\!y.\grad w\partial_{s} w \w{\mathrm{d}}y{\mathrm{d}}s
\underbrace{-\frac{a}{p-1}\int_{s_1}^{s_2}\iint \frac1{s}y.\grad  w w\w\y\s}_{A_1}\nonumber\\
&&\underbrace{-\frac12 \int_{s_1}^{s_2}\!\!\iint\!\!\gamma(s) w^2\w\y\s}_{A_2}
\underbrace{-\frac{a}{p-1}\int_{s_1}^{s_2}\iint \frac{1}{s}\partial_s w w\w \y\s}_{A_3}\\
&&+\underbrace{\int_{s_1}^{s_2}\!\! \frac1{s\sqrt{s}}\!\!\displaystyle\iint\!\!w\partial_s w\w {\mathrm{d}}y{\mathrm{d}}s}_{A_4}+\underbrace{\frac{p+1}{2}\int_{s_1}^{s_2}\!\!\iint e^{-\frac{2(p+1)s}{p-1}}s^{\frac{2a}{p-1}} F_1(\p w)\w {\mathrm{d}}y{\mathrm{d}}s}_{A_5}\nonumber\\
&&+\underbrace{\frac{p+1}{2}\int_{s_1}^{s_2}\!\!\iint e^{-\frac{2(p+1)s}{p-1}}s^{\frac{2a}{p-1}} F_2(\p w)\w {\mathrm{d}}y{\mathrm{d}}s}_{A_6}.\nonumber
\end{eqnarray}
We claim that Proposition \ref{projan22} follows from the following Lemma where we    control  the space-time integral of the nonlinear term  of $w$ and all the terms on the right-hand side of the relation
 (\ref{et44}) in terms of
the left-hand  side:
\begin{lem}\label{g}
 For all  $s\ge 1-\log (T^*(x)-t_3(x_0))$,
for some $t_3(x_0)\in [t_2(x_0), T(x_0))$, for all $\nn >0$, for all $\varepsilon \in  (0,1)$,
\begin{equation}\label{control}
\!\!\iint  |w|^{p+1-\varepsilon}\w {\mathrm{d}}y
\le K _{19}+C
\iint e^{-\frac{2(p+1)s}{p-1}}s^{\frac{2a}{p-1}} F(\p w)\w {\mathrm{d}}y,
\end{equation}
\begin{equation}\label{0control}
\iint e^{-\frac{2(p+1)s}{p-1}}s^{\frac{2a}{p-1}} F(\p w)\w {\mathrm{d}}y
\le   K_{19} +C
\!\!\iint  |w|^{p+1+\varepsilon}\w {\mathrm{d}}y,
\end{equation}
\begin{equation}\label{control3}
\int_{s_1}^{s_2}\!\!\!\iint\!|y.\grad w\partial_s w| \w{\mathrm{d}}y{\mathrm{d}}s
\le \frac{K_{19}}{\nn} +K_{19}\nn
\N,\qquad
\end{equation}
\begin{equation}\label{control1}
\sup_{s\in [s_1,s_2]}\iint \!\!w^2(y,s)\w{\mathrm{d}}y\le \frac{K_{19}}{\nn} +K_{19}\nn\N,
\end{equation}
\begin{equation}\label{control30}
\!\int_{s_1}^{s_2}\!\!\iint\!\!w\partial_{s}w
\frac{|y|^2\w}{1-|y|^2}{\mathrm{d}}y{\mathrm{d}}s
\le \frac{K_{19}}{\nn} +K_{19}\nn
\N,\qquad
\end{equation}
\begin{align}\label{control4}
\iint|w\partial_{s}w|\w {\mathrm{d}}y\le&
\iint (\partial_{s}w)^2\w {\mathrm{d}}y+\frac{K_{19}}{\nn}\nonumber\\ 
&+K_{19}\nn
\N,
\end{align}
\begin{equation}\label{control5}
\iint \Big( (\partial_{s}w(y,s_1))^2+(\partial_{s}w(y,s_2))^2\Big)\w {\mathrm{d}}y
\le K_{19},
\end{equation}
\begin{equation}\label{A10}
|A_1|\le \frac{K_{19}}{\nn} +(K_{19}\nn+
\frac{C}{s_1})
\N,
\end{equation}
\begin{equation}\label{A19}
|A_2|+|A_3|+|A_4|
\le 
\frac{K_{19}}{\nn} +K_{19}\nn\N,
\end{equation}
\begin{equation}\label{A30}
|A_5|+|A_6|
\le C+ \frac{C}{ s_1}\N.
\end{equation}
\end{lem}

\bigskip

Indeed, from (\ref{et44}) and this Lemma,
we deduce that
\begin{align*}
\N\le \frac{K_{19}}{\nn}
+(K_{19}\nn+
 \frac{C}{s_1})
\N.
\end{align*}
Now, we can use the fact that
$s_1\ge -1 -\log(T^*(x)-t_3(x_0))\ge-1 -\log(T(x_0)-t_3(x_0))$ and we choose 
$T(x_0)-t_3(x_0)$  small enough, so that
$ \frac{C}{s_1}\le \frac1{-1 -\log(T(x_0)-t_3(x_0))}\le \frac14 .$
 If we choose  $\nn$ small enough  so that
 $K_{19} \nu_0 \le \frac14$, we obtain
\begin{align*}
\N\le {K_{19}}.
\end{align*}
 
Since $[s,s+1]\subset  [s_1,s_2]$,
 we derive  from \eqref{equiv1} that
 (\ref{projan221}).

\bigskip

It remains to prove Lemma \ref{g}.

\bigskip
%

  {\it{Proof of  Lemma \ref{g}:}}
By \eqref{equiv4}  and \eqref{equiv4bis}, we can write esaily \eqref{control}  and \eqref{0control}.
 Thanks to \eqref{control},  we can adapt with no difficulty
 the proof  in the unperturbed case \cite{MZajm03,MZimrn05}  (up to some very minor changes),
in order to get   the proof of the estimates \eqref{control3}, \eqref{control1}, \eqref{control30}, \eqref{control4}  and \eqref{control5}. 
 Also, by using   \eqref{control} and   the Hardy inequality
\begin{equation}
 \ibint w^2\frac{|y|^2\w}{1-|y|^2}\y\leq C\ibint  |\grad w|^2(1-|y|^2) \w \y+C\ibint w^2\w \y.\nonumber
\end{equation}
 (see the appendix  in \cite{MZajm03} for a proof),  we easily   conclude \eqref{A10} and \eqref{A19}.

\medskip

Finally, it remains only to control the terms $A_5$ and $A_6$.
Note from  \eqref{equiv1},    \eqref{equiv2}  and  \eqref{equiv3}  that
\begin{equation}\label{fev6}
| F_1(\p w)|+| F_2(\p w)|\le   C+C \frac{ F(\p w)}{s}.
\end{equation}
The result \eqref{A30} follows  immediately from  \eqref{fev6}.
This concludes the proof of Lemma \ref{g} and Proposition \ref{projan22} too.
\Box

{{\it {Proof of Theorem \ref{t2}}}} :
Since the derivation of the boundedness of the 
$H^{1}_{loc,u}(\er^N)$ norm of   solution of equation  \eqref{A} from Proposition \ref{projan22} is the same as in Proposition \ref{prop3.1}
(from the estimates
\eqref{feb19}, \eqref{feb191}
\eqref{feb192} and
\eqref{23fev1} (up to some very minor changes).  Moreover, thanks to the estimate 
\eqref{2018N00}, the boundedness of the 
$H^{1}_{loc,u}(\er^N)$ norm,  we  prove  easily  the boundedness of 
$L^{2}_{loc,u}(\er^N)$ norm of  $\partial_s w$ 
 with the ball   $B(0,\frac12 )$.
Thanks to the covering technique (we refer the reader to Merle
 and Zaag \cite{MZimrn05} (pure power case) and Hamza and Zaag in  Lemma 2.8 in \cite{HZjhde12}),  we easily extend this estimate from $B(0,\frac12 )$ to $B$. This concludes the proof of  Theorem \ref{t2}.
\Box
\medskip

\medskip


 \appendix

 \section{Some elementary lemmas.}
 Let $f$, $F$, $F_2$  be the functions defined in  \eqref{deff}, \eqref{defF}  and  \eqref{defF123}.  
Clearly, we have 
\begin{lem}\label{FFF} \ {\  }Let $q>1$,\\ 
\begin{align}
\int_0^u|v|^{q-1}v\log^{{a}}(2+v^2 )\v  \sim&  \frac{| u|^{q+1}}{q+1}\log^{{a}}(2+u^2  ),\quad  \text{ as } \;\; |u| \to \infty,
\label{estF0}\\
F(u)  \sim &\frac{uf(u)}{p+1} \quad  \text{ as } \;\; |u| \to \infty,\label{estF}\\
F_2(u)\sim &\frac{Cuf(u)}{\log^2(2+u^2)}\quad  \text{ as } \;\; |u| \to \infty.\label{estF3}
\end{align}
\end{lem}
\begin{proof} 
 See Lemma A.1 
 in \cite{HZjmaa2020}.
\end{proof}
\Box

Thanks to \eqref{estF0}, \eqref{estF} and \eqref{estF3}, we can
  state and prove  the following estimates: 
\begin{lem}\label{lemm:esth}
For all $s \geq 1$,  for all $z\in \er$,
\begin{align}
 C^{-1} \ps z f(\ps z))\le C+F\left(\ps z)\le C (1+\ps z f(\ps z)\right),\label{equiv1}\\
F_1(\ps z)\le C+C\frac{ \ps z}{s}f(\ps z),\quad\quad\label{equiv2}\\
F_2(\ps z)\le C+C \frac{ \ps z}{s^2}f(\ps z),\quad\quad\label{equiv3}\\
e^{-\frac{2ps}{p-1}}s^{\frac{a}{p-1}}   |f(\ps z)|\le C+ C ( \varepsilon) |z|^{p+\varepsilon  }, \quad\quad \forall  \varepsilon >0,
\label{equiv44}\\
   |z|^{\pp }\le  C ( \varepsilon)e^{-\frac{2ps}{p-1}}s^{\frac{a}{p-1}} | f(\ps z)|+C, \quad\quad \forall  \varepsilon \in (0,p),
\label{equiv44bis}\\
e^{-\frac{2(p+1)s}{p-1}}s^{\frac{2a}{p-1}}  F(\ps z)\le C+  C ( \varepsilon)  |z|^{p+ \varepsilon+1}, \quad\quad \forall  \varepsilon >0,  
\label{equiv4}\\
  |z|^{\pp +1}\le  C ( \varepsilon)
e^{-\frac{2(p+1)s}{p-1}}s^{\frac{2a}{p-1}}  F(\ps z) +C,\quad \quad \forall   \varepsilon \in (0,p+1),
\label{equiv4bis}
\end{align}
where
$\phi$,  $F$, $F_1$ and $F_2$  are given in  \eqref{defphi},    \eqref{defF}, 
 \eqref{defF2} and   \eqref{defF123}.
\end{lem}
\begin{proof} Note that \eqref{equiv1} obviously follows from \eqref{estF}. In order to derive estimates \eqref{equiv2} and \eqref{equiv3}, considering the first case $z^2\ps \geq 4$, then the case $z^2\ps \leq 4$, we would obtain
 \eqref{equiv2} and \eqref{equiv3} by using \eqref{estF0}, \eqref{estF}  and\eqref{estF3}.
Similarly, by taking into account  the  inequality  $\log^a (2+u^2)\le C(\varepsilon )+   |u|^{\varepsilon },$   we  conclude easily 
\eqref{equiv44}, \eqref{equiv44bis}, \eqref{equiv4} and \eqref{equiv4bis}.
 This ends the proof of Lemma \ref{lemm:esth}.\Box
\end{proof}
\def\cprime{$'$} \def\cprime{$'$}
\providecommand{\bysame}{\leavevmode\hbox to3em{\hrulefill}\thinspace}
\providecommand{\MR}{\relax\ifhmode\unskip\space\fi MR }
\providecommand{\MRhref}[2]{%
  \href{http://www.ams.org/mathscinet-getitem?mr=#1}{#2}
}
\providecommand{\href}[2]{#2}


\noindent{\bf Address}:\\
 Department of Basic Sciences, Deanship of Preparatory and Supporting Studies,
  Imam Abdulrahman Bin Faisal University
P.O. Box 1982 Dammam, Saudi Arabia.\\
\vspace{-7mm}
\begin{verbatim}
e-mail:  mahamza@iau.edu.sa
\end{verbatim}
Universit\'e Paris 13, Institut Galil\'ee,
Laboratoire Analyse, G\'eom\'etrie et Applications, CNRS UMR 7539,
99 avenue J.B. Cl\'ement, 93430 Villetaneuse, France.\\
\vspace{-7mm}
\begin{verbatim}
e-mail: Hatem.Zaag@univ-paris13.fr
\end{verbatim}

\end{document}